\documentclass[draft]{amsart}

\usepackage[utf8]{inputenc}
\usepackage[T1]{fontenc} 
\usepackage{amsmath}
\usepackage{amsthm} 
\usepackage{amscd} 
\usepackage{amssymb}
\usepackage{latexsym} 

\usepackage[english]{babel}

\usepackage{hyperref}

\newcommand{\chop}{\dag}
\newcommand{\inter}{\overset{\circ}}
\def\epsilon{\varepsilon}

\renewcommand{\phi}{\varphi}

\DeclareMathOperator{\ind}{\mathrm{ind}}

\newcommand{\Out}{\text{Out}} 
\newcommand{\Aut}{\text{Aut}}
\newcommand{\Stab}{\text{Stab}}

\newcommand{\FN}{F_N}

\newcommand{\CVN}{\text{CV}_N}

\newcommand{\R}{\mathbb R} 
\newcommand{\Z}{\mathbb Z}
 
\newcommand{\N}{\mathbb N}

\def\strutdepth{\dp\strutbox}
\def \ss{\strut\vadjust{\kern-\strutdepth \sss}}
\def \sss{\vtop to \strutdepth{
\baselineskip\strutdepth\vss\llap{$\diamondsuit\;\;$}\null}}

\def\strutdepth{\dp\strutbox}
\def \sst{\strut\vadjust{\kern-\strutdepth \ssss}}
\def \ssss{\vtop to \strutdepth{
\baselineskip\strutdepth\vss\llap{$\spadesuit\;\;$}\null}}

\def\strutdepth{\dp\strutbox}
\def \ssh{\strut\vadjust{\kern-\strutdepth \sssh}}
\def \sssh{\vtop to \strutdepth{
\baselineskip\strutdepth\vss\llap{$\heartsuit\;\;$}\null}}

\def\bar{\overline} 
\def\tilde{\widetilde} 

\newtheorem{thm}{Theorem}[section]
\newtheorem{cor}[thm]{Corollary} 
\newtheorem{lem}[thm]{Lemma}
\newtheorem{prop}[thm]{Proposition}

\newtheorem*{thm*}{Theorem}
\newtheorem*{prop*}{Proposition}
\newtheorem*{thm-main*}{Theorem~\ref{thm:main}}

\theoremstyle{definition} 
\newtheorem{defn}[thm]{Definition}
\newtheorem*{defn*}{Definition} 
\newtheorem{rem}[thm]{Remark}
\newtheorem*{rem*}{Remark}

\numberwithin{equation}{section} 

\begin{document}

\title{Long turns, INP's and 
indices 
for free group automorphisms}

\author{Thierry Coulbois and Martin Lustig}
\address{Institut de mathématiques de Marseille\\ 
Université d'Aix-Marseille\\
39, rue Frédéric Joliot Curie \\
13453 Marseille Cedex 13\\
France\\
\href{mailto:thierry.coulbois@univ-amu.fr}{\nolinkurl{thierry.coulbois@univ-amu.fr}}\\
\href{mailto:martin.lustig@univ-amu.fr}{\nolinkurl{martin.lustig@univ-amu.fr}}
}




\begin{abstract} 
The goal of this paper is to introduce a new tool, called {\em long
  turns}, which is a useful addition to the train track technology
for automorphisms of free groups, in that it allows one to control
periodic INPs in a train track map and hence the index of the induced
automorphism.
\end{abstract}

\maketitle

\section{Introduction}

Automorphisms of a non-abelian free group $\FN$ of finite rank $N \geq
2$ are the 
focal point of many interesting recent research
efforts. The most important class of such automorphisms are the ones
that are {\em irreducible with irreducible powers (iwip)}, also called
{\em fully irreducible} (see \S\ref{preliminaries}). Such
automorphisms can always be represented by a map $f: \Gamma \to
\Gamma$, where $\Gamma$ is a graph equipped with a marking isomorphism
$\FN \overset{\cong}{\longrightarrow} \pi_1(\Gamma)$, and $f$ has the
train track property: the map $f$ defines a {\em gate structure}
$\mathbf G = \mathbf G(f)$ on $\Gamma$, i.e. a partition of the edge
germs at every vertex into {\em gates}, which is preserved by $f$ in
that $f$ maps any {\em $\mathbf G$-legal} path to a $\mathbf G$-legal
image path. Here a path $\gamma$ in $\Gamma$ is {\em $\mathbf
  G$-legal} if at any vertex it never enters and exits through the
same gate.

In this paper we present a new tool, called {\em long turns}, which we
use to define in \S\ref{long-turns} below {\em legalizing} train track
morphisms $g: \Gamma \to \Gamma'$ (with respect to fixed gate
structures $\mathbf G$ on $\Gamma$ and $\mathbf G'$ on $\Gamma'$).
A construction device for such legalizing train track morphisms is given in Section \ref{sec:legalizing-factory}.

\begin{thm}
\label{thm:CL-main}
Let $f: \Gamma \to \Gamma$ be a train track map which represents an
iwip automorphism of $\FN$ and has positive transition matrix $M(f)$.
Let $g: \Gamma \to \Gamma$ be a legalizing train track morphism with
respect to the gate structure $\mathbf G(f)$, and assume that $g$
induces an automorphism on $\pi_1(\Gamma)$ and is gate-stable (see
Definition~\ref{def:gate-stable}).  Then:
\begin{enumerate}
\item\label{concl:main-iwip} The map $f \circ g: \Gamma \to\Gamma$ is a train track
  representative of an iwip automorphism $\phi \in \Out(\FN)$.
\item\label{concl:main-no-inp} There is no periodic INP in $\Gamma$ for the train track map $f
  \circ g$.  In particular there are no non-trivial $(f \circ
  g)$-periodic conjugacy classes in $\FN$.
\item\label{concl:main-index}
The stable index list for $\phi$ is given by the gate index list for $f$.
\end{enumerate}
\end{thm}

To explain the last statement of this theorem we recall from
\cite{GJLL} and \cite{LL3} that for every automorphism $\Phi \in
\Aut(\FN)$ with no periodic $w \in \FN \smallsetminus \{1\}$ the
induced action of a suitable positive power of $\Phi$ on $\partial
\FN$ is particularly simple: There are finitely many attractors and
finitely many repellers on $\partial \FN$, and every other orbit
accumulates positively onto one of the attractors and negatively onto
one of the repellers. The {\em index} of $\Phi$, as defined in
\cite{GJLL}, is in this special case equal to
\[
\ind (\Phi) := \frac{a(\Phi)}{2}-1 \, ,
\] 
where $a(\Phi)$ denotes the number of attractors of the $\Phi$-action
on $\partial \FN$.

For any given $\phi \in \Out(\FN)$ it has been shown in \cite{GJLL}
that up to {\em isogredience} (= conjugation by inner automorphisms),
there are only finitely many lifts $\Phi_k \in \Aut(\FN)$ of $\phi$
which satisfy $\ind(\Phi_k) > 0$. Those indices form the {\em index
  list} of $\phi$ (which is defined up to permutation and typically
given in decreasing order).  Through replacing $\phi$ by a 
suitable 
positive
power $\phi^t$ the number of terms and also their values in the index
list may increase, but eventually it becomes stable 
(with respect to passing to further powers, see \cite{GJLL}
or Section \ref{sec:index-branch-gate} below): 
this is called
the {\em stable index list} of $\phi$. It is an important invariant of
the conjugacy class of $\phi$ in $\Out(\FN)$ with interesting
structural consequences (see \cite{CH12}).  The main result of
\cite{GJLL} states that for the {\em stable index} of any $\phi \in
\Out(\FN)$, i.e. the sum of the values $\ind(\Phi_k)$ in the stable
index list of $\phi$, one has a uniform upper bound:
\begin{equation}\label{eq:indice-def}
\ind_{stab} (\phi) := \sum \ind (\Phi_k) \leq N-1
\end{equation}

On the other hand, the {\em gate index list} for any train track map
$f:\Gamma \to \Gamma$ is the list of the {\em gate indices}
\[
\ind_{\mathbf G(f)}(v_k) = \frac{g(v_k)}{2}-1
\] 
at the essential vertices $v_k$ of $\Gamma$, where $g(v_k)$ is the
number of gates at $v_k$, and $v_k$ is {\em essential} if it has at
least 3 gates and is $f$-periodic.  The advantage of the the gate
index list over the above described index list is that it can be read
off directly from the given train track map.

It turns out that Statement~\ref{concl:main-index} of
Theorem~\ref{thm:CL-main} is a consequence of
Statements~\ref{concl:main-iwip} and \ref{concl:main-no-inp}. This
follows from standard results on Outer space and $\R$-trees. For the
convenience of the reader we recall and assemble the relevant facts in
Sections~\ref{sec:index-branch-gate}.

It is a direct consequence of Nielsen-Thurston theory for surface
homeomorphisms that in the special case where $\phi$ is {\em
  geometric}, i.e. $\phi$ is induced by a homeomorphism of some
surface with boundary, the above stable index
inequality~(\ref{eq:indice-def}) becomes an equality.  In general,
however, computer experiments of the first author (see
\cite[subsection~7.3]{CLP}, and \cite{C15}) indicate that random
automorphisms have very low stable index; indeed, for up to rank $N
\leq 9$ more than half of the investigated automorphisms have index
list equal to $[\frac{1}{2}]$, to $[\frac{1}{2}, \frac{1}{2}]$, to
$[1]$ or to $[\frac{1}{2}, \frac{1}{2}, \frac{1}{2}]$.

In our subsequent work~\cite{CLP} we use Theorem~\ref{thm:CL-main} as
crucial tool to produce explicit automorphisms $\phi_I \in \Out(\FN)$
which realize as stable index list any given list $I = [j_1, \ldots,
  j_\ell]$ of positive values $j_k \in \frac{1}{2}\Z$ which satisfies
the above inequality~(\ref{eq:indice-def}), thus answering a question
posed by Handel and Mosher~\cite{HM}.

Theorem~\ref{thm:CL-main} is also used to derive in
Corollary~\ref{monoid} information about the elements of the monoid
generated by train track morphisms on a common graph with respect to a
fixed gate structure. This result is also the ``door opener'' to a
further study of {\em strata} in Outer space, in analogy to the well
known and heavily investigated strata in Teichm\"uller space, see for
instance \cite{MS}.

\smallskip
\noindent {\em Acknowledgements:} This paper came into existence
largely due to Catherine Pfaff's postdoc stay in Marseille: It is part
of a larger project on which we started working with her during the
last months of her position in Marseille. Although the content of our
discussions mostly concern the companion paper~\cite{CLP} (which was
indeed meant to be joint work with her), some of the material
presented in this paper as well must surely be influenced by having
talked to Catherine.

We would also like to thank Ilya Kapovich for some valuable comments
on indices of random automorphisms, in the context of his recent work
\cite{KP} with Pfaff.

\section{Notation and conventions}
\label{preliminaries}

Throughout this paper, $\FN$ will denote the non-abelian free group of
finite rank $N \geq 2$, and $\Out(\FN)$ the group of its outer
automorphisms.  Furthermore, we will use the following conventions and
notations:

A {\em graph} is always connected, without vertices of valence
$1$, and finite, unless it is the universal covering $\tilde \Gamma$ of
a finite connected graph $\Gamma$. For every oriented edge $e$ of
$\Gamma$ we denote by $\bar e$ the edge with reversed orientation. The
set $E^{\pm}(\Gamma)$ denotes the set of all edges $e$ including their
``inverses'' $\bar e$, while by $E^+(\Gamma)$ we mean a section
(sometimes called an {\em orientation}) of the quotient map $e \mapsto
\{e, \bar e\}$ on $E^{\pm}(\Gamma)$.

An {\em edge path} $\gamma$ in $\Gamma$ is a (possibly infinite)
sequence $\gamma = \ldots e_{i-1} e_i e_{i+1} \ldots$ of edges $e_i
\in E^\pm(\Gamma)$ where the initial vertex of any $e_i$ must agree with
the terminal vertex of $e_{i+1}$. Such an edge path $\gamma$ is {\em
  reduced} if it doesn't contain any {\em backtracking subpath}
$\gamma'$, i.e. $\gamma'$ is a finite sub-edge-path which has
coinciding initial and terminal vertex, and which is contractible
relative to its endpoints. If a finite edge path $\gamma$ is not reduced,
then it is always homotopic rel. endpoints in $\Gamma$ to a reduced
edge path $[\gamma]$, and this homotopy can be expressed as iterative
contraction of backtracking subpaths.

The {\em combinatorial length} $|\gamma|$ of a finite edge path
$\gamma$ is equal to the number of edges traversed by $\gamma$, and
$\gamma$ is {\em trivial} if $|\gamma| = 0$.

A {\em graph map} is a map $f: \Gamma \to \Gamma'$ between graphs
which maps vertices to vertices and edges $e$ to (not necessarily
reduced) edge paths $f(e)$.  If $f$ {\em has no contracted edges},
i.e. none of the image paths $f(e)$ is trivial, then $f$ induces a
{\em differential} $Df: E^\pm(\Gamma) \to E^\pm(\Gamma')$ which maps
$e$ to the initial edge of $f(e)$.

Two edges $e$ and $e'$ with same initial vertex $v$ form a {\em turn}
$(e, e')$ at $v$, and the map $Df$ induces a map $D^2f:
(e, e') \mapsto (Df(e), Df(e'))$ from the turns of $\Gamma$ to the
turns of $\Gamma'$. The turn $(e, e')$ is {\em degenerate} if we have
$e = e'$ in $E^\pm(\Gamma)$.

We say that an edge path $\gamma = \ldots e_{i-1} e_i e_{i+1} \ldots$
{\em crosses over} (or {\em uses}) a turn $(e, e')$ if for some index
$i$ one has $e_i = \bar e$ and $e_{i+1} = e'$.

To every graph map $f: \Gamma \to \Gamma'$ there is associated a
non-negative {\em transition matrix} $M(f) = (m_{e', e})_{e'\in
  E^+(\Gamma'), \, e\in E^+(\Gamma)}$, where the coefficient $m_{e',
  e}$ is the number of times that the (possibly unreduced) edge path
$f(e)$ crosses over the edge $e'$ or the inversely oriented edge $\bar
e'$ (both occurrences counted positively!).  It follows directly from
the definition that for any two graph maps $f: \Gamma \to \Gamma'$ and
$g: \Gamma' \to \Gamma''$ one has:
\[
M(g\circ f) = M(g) M(f).
\]

Recall that a non-negative square matrix $M$ is {\em primitive} if
there is a positive power $M^t$ with positive coefficients only.

An automorphism $\phi \in \Out(\FN)$ is called {\em irreducible with
  irreducible powers (iwip)} or {\em fully irreducible} if there is no
proper free factor of $\FN$ which is fixed up to conjugacy by any
positive power of $\phi$.

A graph map $f: \Gamma \to \Gamma$ is called {\em expanding} if for
every edge $e \in E^\pm(\Gamma)$ there is an exponent $t \geq 1$ such
that $|f^t(e)| \geq 2$.  If $f$ represents an iwip automorphism, then,
up to passing to a quotient graph through contracting $f$-invariant
subtrees in $\Gamma$, the map $f$ must be expanding.

\section{Gate structures for graphs}\label{sec:gates}

\begin{defn}
\label{def:gates}
(1) A {\em gate structure} $\mathbf G_v$ at a vertex $v$ of a graph
$\Gamma$ is a partition of the edges $e \in E^\pm(\Gamma)$ with
initial vertex $v$ into equivalence classes, called {\em gates}.

\smallskip
\noindent
(2) A gate structure $\mathbf G$ on $\Gamma$ is the collection of gate
structures $\mathbf G_v$ for every vertex $v$ of $\Gamma$. A graph
$\Gamma$ together with a gate structure $\mathbf G$ on $\Gamma$ has
been termed a {\em train track} in \cite{Lu_conj-iwips}.

\smallskip
\noindent
(3) Two edges $e, e' \in {E}^\pm(\Gamma)$ with same initial vertex $v$
form a {\em legal turn} $(e, e')$ (with respect to $\mathbf G$) if $e$
and $e'$ belong to distinct gates. Otherwise the turn $(e, e')$ is
called {\em illegal}.

\smallskip
\noindent
(4) A (finite or infinite) edge path $\gamma = \ldots e_{i-1}e_i
e_{i+1} \ldots$ in $\Gamma$ is called {\em legal} (with respect to
$\mathbf G$), if all of the turns $(\bar e_i, e_{i+1})$ over which
$\gamma$ {crosses} are legal.
\end{defn}

These notions, introduced in \cite{Lu_conj-iwips} and in a similar
fashion elsewhere (\cite{BF}, \cite{H}, \cite{HM}, ...), have been
inspired by the fundamental paper \cite{BH}, where the gate structure
$\mathbf G = \mathbf G(f)$ is defined through a train track map $f:
\Gamma \to \Gamma$, as explained below in
Definition~\ref{def:intrinsic} and Remark~\ref{classic-tts}.  There
are other natural occurrences of gate structures, for example given by
an edge-isometric $\FN$-equivariant map from the universal cover
$\tilde \Gamma$ to an $\R$-tree $T$, see \cite{Lu_conj-iwips}.

\begin{defn}
\label{train-track-morphism}
Let $\Gamma$ and $\Gamma'$ be graphs equipped with gate structures
$\mathbf G$ and $\mathbf G'$ respectively.  A graph map $f: \Gamma \to
\Gamma'$ is called a {\em train track morphism} if the following two
conditions hold:
\begin{enumerate}
\item $f$ has no contracted edges.
\item\label{def-cond:train-track} $f$ has the {\em train track property}: it maps legal paths to
  legal paths.
\end{enumerate}
\end{defn}

\begin{rem}
\label{tt-morphisms-again}
The reader can verify directly that condition~(\ref{def-cond:train-track}) of
Definition~\ref{train-track-morphism} is equivalent to the following
two more local conditions:
\begin{enumerate}
\item[(2'a)]
the map $D^2 f$ (see Section~\ref{preliminaries}) maps legal turns to legal turns, and
\item[(2'b)]
for every edge $e$ of $\Gamma$ the edge path $f(e)$ is legal.
\end{enumerate}
\end{rem}

\begin{rem}
\label{composition}
(a) It follows directly from Definition~\ref{train-track-morphism}
that the composition $g \circ f$ of two train track morphisms $f:
\Gamma \to \Gamma'$ and $g: \Gamma' \to \Gamma''$, with respect to the
same gate structure on $\Gamma'$, is again a train track morphism.

\smallskip
\noindent
(b) In particular, in the special case $\Gamma' = \Gamma$ and $\mathbf
G' = \mathbf G$ we note that for any edge $e$ and any integer $t \geq 0$
the edge path $f^t(e)$ is legal and hence reduced.
\end{rem}

\begin{rem}
\label{finer-coarser}
If $\Gamma$ is equipped with two gate structures $\mathbf G_1$ and
$\mathbf G_2$, such that every gate of $\mathbf G_2$ is contained in a
gate of $\mathbf G_1$, then $\mathbf G_2$ is {\em finer} than (or a
{\em refinement} of) $\mathbf G_1$, while $\mathbf G_1$ is {\em
  coarser} than ${\mathbf G_2}$.

Any train track morphism $f: \Gamma \to \Gamma'$ stays a train track
morphism if the gate structure of $\Gamma$ is replaced by a coarser
one, or the gate structure of $\Gamma'$ by a finer one.
\end{rem}

In order to pinpoint certain subtleties which will trouble us later, we define:

\begin{defn}
\label{gate-morphism}
(a)
Let $\Gamma$ and $\Gamma'$ be graphs equipped with gate structures
$\mathbf G$ and $\mathbf G'$ respectively.  A graph map $f: \Gamma \to
\Gamma'$ is called a {\em gate structure morphism} if the following two
conditions hold:
\begin{enumerate}
\item $f$ has no contracted edges.
\item The induced map $D f$ (see Section~\ref{preliminaries}) maps all
  edges in any gate of $\mathbf G$ to a common {\em image gate} of
  $\mathbf G'$. In other words: $f$ induces a well defined map $f_{\mathbf
  G}: \mathbf G \to \mathbf G'$.
\end{enumerate}

\smallskip
\noindent
(b)
If $f$ is not necessarily a gate structure morphism on all of
$\Gamma$, but induces a well defined map $f_{\mathbf G_v}: \mathbf G_v \to
\mathbf G'_{f(v)}$ on the gates at a given vertex $v$ of $\Gamma$, we
say that $f$ is a {\em local gate structure morphism at $v$}.
\end{defn}

We'd like to alert the reader that the notions of ``train track
morphisms'' and ``gate structure morphisms'' are sort of perpendicular
to each other, as neither of them implies the other: A gate structure
morphism is a train track morphism $f$ if and only if the induced map
$f_\mathbf G$ is injective and the image of every edge is legal. On the
other hand, a general train track morphism may be quite far from being
a gate structure morphism, as the identity map which passes from given
gates structure to a strict refinement (compare
Remark~\ref{finer-coarser}) shows.  However, for the special case of
graph self-maps we have the following observation which will turn out
to be often quite useful:

\begin{lem}
\label{periodic-gates}
Let the graph self-map $f: \Gamma \to \Gamma$ be a train track
morphism with respect to some gate structure $\mathbf G$ on $\Gamma$.

Then at every $f$-periodic vertex $v$ of $\Gamma$ the induced map $D^2
f$ maps legal to legal and illegal to illegal turns.  In particular,
$f$ induces a well defined bijective map from the gates at $v$ to the
gates at $f(v)$. In other words:

For every periodic vertex $v$ the map $f$ induces a local gate
structure morphism $f_{\mathbf G_v}: \mathbf G_v \to \mathbf G_{f(v)}$ which
is bijective.
\end{lem}

\begin{proof}
Since $f$ has the train track property, the induced map $D f$ maps
edges in distinct gates to edges in distinct gates.  Thus for every
vertex $v$ of $\Gamma$ the number of gates at $f(v)$ must be larger or
equal to the number of gates at $v$. Furthermore, if the two numbers
are equal, it follows directly that edges in any given gate must be
mapped by $D f$ to edges that also lie all in a common gate. This
implies directly the claimed statement for periodic vertices of a
train track self-morphism.
\end{proof}

\medskip

The most important special case of graph maps, and also the source of
the notion of ``gates'', is the case of a self-map $f: \Gamma \to
\Gamma$.  It turns out that self-maps which don't even have the train
track property define already a gate structure on $\Gamma$:

\begin{defn}
\label{def:intrinsic}
Let $f: \Gamma \to \Gamma$ be a graph self-map with no contracted
edges. The {\em intrinsic} gate structure $\mathbf G = \mathbf G(f)$
is defined by declaring edges $e, e'$ with same initial vertex to
belong to the same gate if and only if some power of $f$ maps $e$ and
$e'$ to edge paths which have non-trivial initial subpaths in common.
\end{defn}

\begin{rem}
\label{built-in}
(1) This definition has built in that $f$ preserves the gates of the
gate structure $\mathbf G(f)$:
\begin{enumerate}
\item[(a)] edges in distinct gates at a common vertex are mapped into distinct gates and,
\item[(b)] all edges in a given gate are mapped into the same gate.
\end{enumerate}
In other words: The graph self-map $f$ is a gate structure morphism
with respect to $\mathbf G(f)$, and the induced map $f_{\mathbf G(f)}$ is
injective.

\smallskip
\noindent
(2) Hence, by Remark~\ref{tt-morphisms-again}, in order to check
whether $f$ is a train track morphism with respect to $\mathbf G(f)$, we
only need to check that for every edge $e$ the edge path $f(e)$ is
legal.
\end{rem}

\begin{rem} 
\label{classic-tts}
(1) Recall that a {\em classical train track map} as introduced by
Bestvina and Handel~\cite{BH} is a graph self-map $f: \Gamma \to
\Gamma$ which has the property that for any edge $e$ and any integer
$t \geq 1$ the edge path $f^t(e)$ is reduced.

We have already noted in Remark~\ref{composition}~(b) that any train
track self-morphism $f: \Gamma \to \Gamma$, with respect to {\em any}
gate structure $\mathbf G$ on $\Gamma$, is such a classical train track
map.

Conversely, it follows directly from the above definitions that every
classical train track map is a train track morphism with respect to
the intrinsic gate structure $\mathbf G(f)$.

\smallskip
\noindent
(2) However, the reader should be warned that for any classical train
track map $f$ as above, in addition to $\mathbf G(f)$ there may well
be other gate structures $\mathbf G$, with respect to which $f$ is
also a train track morphism: For example every positive automorphism
is represented by a train track morphism on the rose with respect to
the gate structure at the sole vertex which consists only of the
``positive'' and of the ``negative'' gates.

\smallskip
\noindent
(3) On the other hand, it follows directly from the above definitions
that any other such gate structure $\mathbf G$ must be coarser than
$\mathbf G(f)$, so that $\mathbf G(f)$ is indeed the finest gate structure
with respect to which the train track map $f$ is a train track
morphism.
\end{rem}

\medskip

The following turns out to be a useful notion for the sequel. In order
to properly state it we recall from Lemma~\ref{periodic-gates} that
any train track self-map $f$ which fixes a vertex $v$ induces a well
defined bijection $f_{\mathbf G_v}$ on the set of gates at that vertex.

\begin{defn}
\label{def:gate-stable}
A train track morphism $f: \Gamma \to \Gamma$ with respect to some
gate structure $\mathbf G$ is called {\em gate-stable} if $f$ fixes
every vertex of $\Gamma$, and at every vertex $f$ fixes also every
gate of $\mathbf G$.
\end{defn}

We notice directly that every train track self-morphism which acts
periodically on every vertex possesses a positive power that is
gate-stable.

\medskip

We say that a path $\gamma$ {\em crosses over a gate turn $(\mathfrak g_i,
  \mathfrak g_j)$} if $\gamma$ contains the subpath $\bar e\cdot e'$
with $e \in \mathfrak g_i$ and $e' \in \mathfrak g_j$.

\begin{defn}
\label{gate-index}
  Let $f: \Gamma \to \Gamma$ be a train track morphism with respect to
  some gate structure $\mathbf G$ on $\Gamma$. 

\smallskip
\noindent
(a) For every vertex $v$ of $\Gamma$ we define the {\em
  gate-Whitehead-graph} $Wh_{\mathbf G}^v(f)$ to be the graph with the
set $\mathbf G_v$ of gates at $v$ as vertex set, and with a
(non-oriented) edge connecting $\mathfrak g_i$ to $\mathfrak g_j$ if
for some $t \geq 1$ and some edge $e \in E^\pm(\Gamma)$ the path
$f^t(e)$ crosses over $v$ entering through $\mathfrak g_i$ and leaving
through $\mathfrak g_j$.

\smallskip
\noindent
(b) A vertex $v$ in $\Gamma$ is {\em essential} if it is periodic and
if there are at least three gates at $v$. The {\em gate index} at $v$
is defined as
\[
\ind_\mathbf G(v) := \frac{g(v)}{2} - 1
\]
where $g(v)$ denotes the number of gates at $v$.

\smallskip
\noindent
(c) The {\em gate index list} for $f$ is the list of gate indices of
$\mathbf G(f)$ at essential vertices. We usually order such a list as
decreasing sequence of its values.
\end{defn}

For graph self-maps the Whitehead graphs at the vertices (in various
dialects) have been used previously (e.g. see
\cite{BH,HM,JL,KL,Lu_conj-iwips}).

\begin{prop}
\label{Wh-composition}
Let $f$ and $g$ be train track morphisms of a graph $\Gamma$ with
respect to the gate structure $\mathbf G := \mathbf G(f)$. If $g$
induces an automorphism of $\pi_1(\Gamma)$ and is gate-stable, then
for every vertex $v$ of $\Gamma$ the graph $Wh_{\mathbf G}^v(f)$ is a
subgraph of both, $Wh_{\mathbf G}^v(f\circ g)$ and
$Wh_{\mathbf G}^v(g\circ f)$.
\end{prop}

\begin{proof}
  We first use the hypothesis $\mathbf G = \mathbf G(f)$ to deduce
  (see Remark~\ref{built-in} (1)) that $f$ is a gate structure
  morphism and hence induces a map
  $f_\mathbf G: \mathbf G \to\mathbf G$. Hence there is a well defined
  map $D^2_{\mathbf G} f$ on the gate turns of $\Gamma$, given by
  $D^2_{\mathbf G} f(\mathfrak g_i, \mathfrak g_j) := (f_\mathbf
  G(\mathfrak g_i), f_\mathbf G(\mathfrak g_j))$.

  For any vertex $v$ of $\Gamma$ it follows from the definition of the
  gate-Whitehead-graph that in $Wh_{\mathbf G}^v(f)$ two ``vertices''
  $\mathfrak g_i$ and $\mathfrak g_j$ are connected by an edge if and
  only if one of the following occurs:

  (a) For some edge $e \in E^\pm(\Gamma)$ the path $f(e)$ crosses over
  the gate turn $(\mathfrak g_i, \mathfrak g_j)$.

  (b) For some $t \geq 1$ the map $D^2_{\mathbf G} f^t$ maps one of
  the gate turns crossed over by some $f(e)$ to the gate turn
  $(\mathfrak g_i, \mathfrak g_j)$.

  From the hypothesis that $g$ is a homotopy equivalence and our
  convention that graphs don't have valence 1 vertices we deduce that
  each edge $e$ appears in the image $g(e')$ of some edge $e'$. It
  follows that all gate turns $(\mathfrak g_i, \mathfrak g_j)$ crossed
  over by the path $f(e)$ are also crossed over by the path
  $f\circ g(e')$.  Since $g$ is gate-stable, the maps $g_\mathbf G$
  and $D^2_{\mathbf G} g$ are well defined, and $D^2_{\mathbf G} g$
  acts as the identity on the set of gate turns.  This implies both,
  that $g\circ f(e)$ crosses over the same gate turns as $f(e)$, and
  that the above property (b) for $f$ is equivalent to property (b)
  for $f\circ g$ or $g\circ f$.

  This shows that, for every vertex $v$ of $\Gamma$, in the graph
  $Wh_{\mathbf G}^v(f)$ two ``vertices'' $\mathfrak g_i$ and
  $\mathfrak g_j$ are connected by an edge only if the same
  ``vertices'' are also connected by an edge in
  $Wh_{\mathbf G}^v(f \circ g)$ and in $Wh_{\mathbf G}^v(g \circ f)$.
\end{proof}

\section{In the absence of INPs}

The notion of an INP is a classical concept, going back to \cite{BH}.
Origin\-ally, ``INP'' was an abbreviation for ``indivisible Nielsen
path''; however, through frequent use it has become mainly an acronym,
and as such we treat it here:

\begin{defn}
  Let $f: \Gamma \to \Gamma$ be a train track morphism with respect to
  some gate structure $\mathbf G$ on $\Gamma$. A reduced path $\eta$
  in $\Gamma$ is called a {\em periodic INP} if
  $\eta = \gamma^{-1} \circ \gamma'$, where the two {\em branches}
  $\gamma$ and $\gamma'$ are non-trvial legal paths, and $f^t(\eta)$
  is homotopic relative endpoints to $\eta$, for some $t \geq 1$.
\end{defn}

The reader should be alerted that the endpoints of a periodic INP
$\eta$ may well not be vertices of $\Gamma$ (although this can be
readily achieved by subdividing edges and their iterated $f$-images,
which is a finite procedure since the endpoints of $\eta$ are by
definition $f$-periodic).

\smallskip


The importance of the gate-Whitehead-graph (see
Definition~\ref{gate-index}) for a train track map is underlined by
the following ``irreducibility criterion'' from \cite{JL}. We quote
here only a simplified version which is used below in
Section~\ref{legalizing-maps}:

\begin{prop}[{\cite[Proposition 5.1]{JL}}]
\label{JL-quote}
Let $f: \Gamma \to \Gamma$ be a train track representative of $\phi
\in \Out(\FN)$. Assume furthermore:
\begin{enumerate}
\item\label{cond:primitive} The transition matrix $M(f)$ is primitive.
\item\label{cond:gwg-connected} The gate-Whitehead-graph $Wh_{\mathbf G(f)}^v(f)$ for $f$ at
  every vertex $v$ of $\Gamma$ is connected.
\item\label{cond:no-inp} There is no periodic INP for $f$ in $\Gamma$.
\end{enumerate}
Then $\phi$ is iwip (= fully irreducible).
\qed
\end{prop}

A ``partial converse'' of this result is given by the following:

\begin{prop}[{\cite[Proposition~5.1]{JL}}]
\label{prop:no-iwp-implies-irred}
  Let $f:\Gamma\to\Gamma$ be a train track representative of some iwip
  automorphism $\phi\in\Out(\FN)$.  Assume that $f$ is expanding and
  that there is no periodic INP for $f$. 
  
  Then $f$ must satisfy conditions~(\ref{cond:primitive}) and
  (\ref{cond:gwg-connected}) from Proposition~\ref{JL-quote}.  \qed
\end{prop}

\begin{rem}
\label{proof-of-partial-converse}
Since the connectedness of the gate-Whitehead-graph for $f$ at every
vertex is a direct consequence of the connectedness of the classical
Whitehead graph of $f$ at $v$, one can obtain
Proposition~\ref{prop:no-iwp-implies-irred} also as direct consequence
of the classical known irreducibility criterion (see \cite{Ka2013}).

Alternatively, the reader may prefer to go for a direct proof of
Proposition~\ref{prop:no-iwp-implies-irred} according to the following
lines: If $M(f)$ is not primitive, then the transition matrix $M(f^t)$
of a positive power of $f$ must be reducible, so that there is an
$f^t$-invariant subgraph of $\Gamma$, which by expansiveness of $f$
must have as fundamental group a proper free factor of $\FN$.

Similarly, if $Wh_{\mathbf G(f)}^v(f)$ is not connected for some
vertex $v$ of $\Gamma$, then we can introduce ``invisible edges'' to
blow up that vertex and thus again find, as complement of the
invisible edges, an invariant subgraph with a proper free factor of
$\FN$ as fundamental group.  (More details regarding the blow-up
technology of vertices by means of invisible edges can be found in the
proof of \cite[Proposition~5.1]{JL}, or in the description before
Lemma~6.6 of \cite{KL}.)
\end{rem}

\section{Long turns}
\label{long-turns}

Recall that every graph map $f:\Gamma \to \Gamma'$ which induces an
isomorphism on $\pi_1 \Gamma$ (i.e. $f$ is a homotopy equivalence)
possesses a cancellation bound which can be expressed topologically as
a {\em bounded backtracking constant} of any lift
$\tilde f:\tilde \Gamma \to \tilde \Gamma'$ to the universal coverings
(see \cite{GJLL}): There exists a constant $C(f) \geq 0$ such that for
any reduced edge path $\gamma$ in $\tilde \Gamma$ the unreduced image
path $\tilde f(\gamma)$ is contained in the $C(f)$-neighborhood of the
reduced path $[\tilde f(\gamma)]$, for the metric on $\tilde \Gamma'$
defined by combinatorial path length (see \S\ref{preliminaries}).  It
is known (see \cite{BFH,GJLL}) for homotopy equivalences $f$ that the
{\em combinatorial volume} of $f(\Gamma)$, i.e. the total length of
the edge paths $f(e_i)$ for all $e_i \in {\rm E}^+(\Gamma)$, is such a
cancellation bound.

\begin{defn}
\label{longturns}
Let $\Gamma$ and $\Gamma'$ be graphs equipped with gate structures,
and let $f: \Gamma \to \Gamma'$ be a train track morphism.

\smallskip
\noindent
(1) A {\em long turn} $(\gamma,\gamma')$ in $\Gamma$ is given by two
non-trivial legal edge paths $\gamma$ and $\gamma'$ with common
initial vertex of $\Gamma$ and distinct first edges.

\smallskip
\noindent
(2) A long turn $(\gamma,\gamma')$ is called {\em illegal} if the
initial edges of $\gamma$ and $\gamma'$ form an illegal turn (in the
sense of Definition~\ref{def:gates} (3)). Otherwise $(\gamma,\gamma')$
is called {\em legal}.

\smallskip
\noindent
(3) A long turn $(\gamma_1, \gamma'_1)$ in $\Gamma'$ is called {\em
  the $f$-image} of the long turn $(\gamma, \gamma')$ in $\Gamma$,
denoted by
\[
f^{LT}(\gamma,\gamma') := (\gamma_1,\gamma'_1)\, ,
\]
if $f(\gamma) = \gamma_0 \circ \gamma_1$ and $f(\gamma') = \gamma'_0
\circ \gamma'_1$ such that $\gamma_0 = \gamma'_0$ is the maximal
common initial subpath of $f(\gamma)$ and $f(\gamma')$, and $\gamma_1$
and $\gamma'_1$ are non-trivial paths.

A long turn which possesses an $f$-image is called an {\em $f$-long
  turn}. Otherwise we call it {\em $f$-degenerate}.
\end{defn}

Note that a turn $(e, e')$ in the classical sense of
Definition~\ref{def:gates} (3) is in particular a long turn, except in
the particular case where it is {\em degenerate}, i.e. $e = e'$.

\begin{rem}
\label{short-long-turns}
We note that the existence of a cancellation bound $C(f)$ for a train
track morphism $f: \Gamma \to \Gamma'$ together with the ``no
contracted edges'' assumption in Definition~\ref{train-track-morphism}
imply that every long turn $(\gamma, \gamma')$ with $|\gamma|,
|\gamma'| \geq C(f) +1$ is $f$-long.
\end{rem}

We now want to consider the composition $g \circ f$ of two train
track morphisms $f:\Gamma \to \Gamma'$ and $g:\Gamma' \to \Gamma''$. We observe:

\smallskip
\noindent
(1)
If a long turn $(\gamma_1, \gamma_2)$ in $\Gamma$ is $f$-degenerate,
then $f(\gamma_1)$ is an initial subpath of $f(\gamma_2)$ (or
conversely), and hence, since $\gamma_1, \gamma_2$ as well as their
images are legal, $g\circ f(\gamma_1)$ is an initial subpath of $g\circ f(\gamma_2)$
(or conversely), so that $(\gamma_1, \gamma_2)$ is also $(g\circ
f)$-degenerate.

\smallskip
\noindent
(2) Similarly, if $(\gamma_1, \gamma_2)$ is $f$-long, but its
$f$-image long turn $(\gamma'_1, \gamma'_2)$ is $g$-degenerate, it
follows that $(\gamma_1, \gamma_2)$ is also $(g\circ f)$-degenerate.

\smallskip
\noindent
(3) On the other hand, if $(\gamma_1, \gamma_2)$ is $f$-long and its
$f$-image long turn $(\gamma'_1, \gamma'_2)$ is $g$-long, then we see
directly that $(\gamma_1, \gamma_2)$ is also $(g\circ f)$-long, and
that its $(g\circ f)$-image coincides with the $g$-image of
$(\gamma'_1, \gamma'_2)$. 

We summarize:

\begin{lem}
\label{product-map}
Let $f:\Gamma \to \Gamma'$ and $g:\Gamma' \to \Gamma''$ be two train
track morphisms (with respect to the same gate structure on
$\Gamma'$).

A long turn $(\gamma_1, \gamma_2)$ in $\Gamma$ is $(g\circ f)$-long if
and only if $(\gamma_1, \gamma_2)$ is $f$-long and its $f$-image long
turn $f^{LT}(\gamma_1, \gamma_2)$ is $g$-long. In this case one has:
\[
(g\circ f)^{LT}(\gamma_1, \gamma_2) = g^{LT}(f^{LT}(\gamma_1, \gamma_2))
\]
\qed
\end{lem}

\begin{rem}
\label{legal-long}
If a long turn $(\gamma, \gamma')$ in $\Gamma$ is legal, then it is
$f$-long for every train track morphism $f: \Gamma \to \Gamma'$, and
the $f$-image is again legal (indeed, the common initial subpath
$\gamma_0 = \gamma'_0$ from part (3) of Definition~\ref{longturns} is
in this case trivial).

On the other hand, for an illegal $f$-long turn the $f$-image can be
either legal or illegal; both cases do occur.
\end{rem}

\begin{rem}
\label{subturns}
(1) Any long turn $(\gamma'_1, \gamma'_2)$ in $\Gamma$ is called a
{\em subturn} of a long turn $(\gamma_1, \gamma_2)$ if each
$\gamma'_i$ is an initial subpath of $\gamma_i$, for $i=1,2$. If both
long turns, $(\gamma_1, \gamma_2)$ and $(\gamma'_1, \gamma'_2)$, are
$f$-long for some train track map $f$, then their $f$-images are
either both legal or both illegal. Hence, in order to test legality of
the $f$-image of any long turn, it suffices to calculate the $f$-image
of the shortest $f$-long subturn of the given long turn.

\smallskip
\noindent
(2) If the {\em length} of a long turn $(\gamma_1, \gamma_2)$, defined
as $\min\{|\gamma_1|, |\gamma_2|\}$, satisfies
$\min\{|\gamma_1|, |\gamma_2|\} \geq C \geq 0$, then we will denote
below the subturn $(\gamma'_1, \gamma'_2)$ of $(\gamma_1, \gamma_2)$
with $|\gamma'_1| = |\gamma'_2|= C$ by:
\[
(\gamma_1, \gamma_2)\chop^C := (\gamma'_1, \gamma'_2)
\]
We also denote the set of all long turns $(\gamma_1, \gamma_2)$ in
$\Gamma$ with {\em branch length} $|\gamma_1| = |\gamma_2|= C$ by
$LT_C(\Gamma)$.
\end{rem}

\begin{defn}
\label{legalizing}
A train track morphism $f: \Gamma \to\Gamma'$ is called {\em
  legalizing} if every illegal $f$-long turn has legal $f$-image.
\end{defn}

It follows from Remarks~\ref{short-long-turns} and \ref{subturns} that
a train track map $f$ with cancellation bound $C(f)$ is legalizing if
and only if every long turn built from legal paths of length $C(f) +1$
has legal $f$-image. Indeed, it follows from
Corollary~\ref{INP-alternative} below that many (or even ``most'')
train track representatives of iwip automorphisms are legalizing.

\begin{rem}
\label{warning}
The reader should be warned that in the (frequently occurring) case
that a graph map $f: \Gamma \to \Gamma'$ is a train track morphism
with respect to a gate structure $\mathbf G$ on $\Gamma$ and
simultaneously with respect to a coarser gate structure $\mathbf G_0$,
then $f$ may well be legalizing with respect to $\mathbf G$ but not
with respect to $\mathbf G_0$.
\end{rem}

We should perhaps point out here that although every legalizing train
track morphism maps every path $\gamma$ with a single illegal turn to
a legal path $[f(\gamma)]$ (after reduction!), the same does not at
all follow for a path $\gamma$ with more than one illegal turn.  The
only conclusion one can draw is that the number of illegal turns in
the reduced image path $[f(\gamma)]$ is at most half times the number
of illegal turns in $\gamma$ plus $1$.

This notion of a ``legalizing map'' is robust and easy to handle, as
is illustrated by following:

\begin{prop}
\label{composition-legalizing}
Let $f: \Gamma_1 \to \Gamma_2$ and $g: \Gamma_2 \to \Gamma_3$ be two
train track morphisms. If either $f$ or $g$ is legalizing, then the
composition $g \circ f$ is a legalizing train track morphism.
\end{prop}

\begin{proof}
  This follows directly from Lemma~\ref{product-map}, together with
  the observation in Remark~\ref{legal-long} that for any train track
  morphism and any legal long turn the image long turn is always
  legal.
\end{proof}

The following observation will be used crucially in the next section:

\begin{lem}
\label{lem:legalizing-implies-intrinsec}
(a) Let $g: \Gamma \to \Gamma$ be a train track morphism with respect
to some gate structure $\mathbf G$, and assume that $g$ is both,
legalizing and gate-stable. Then the gate structure $\mathbf G$ is
equal to the intrinsic gate structure of $g$:
\[
\mathbf G=\mathbf G(g)
\]

\smallskip
\noindent
(b) Moreover, if $f: \Gamma \to \Gamma$ is another train track
morphism with respect to $\mathbf G$, then the intrinsic gate
structure $\mathbf G(f \circ g)$ of the composition $f\circ g$ is
equal $\mathbf G$.
\end{lem}

\begin{proof}
  (a) From the hypothesis that $g$ is legalizing we know that the
  $g$-image of every $g$-long illegal turn $(\gamma, \gamma')$ is a
  legal long turn. But since $g$ is gate-stable, the initial edges of
  $g(\gamma)$ and $g(\gamma')$ must lie in the same gate, so that they
  cannot form a legal turn. Hence they must belong to the common
  initial subpath of $g(\gamma)$ and $g(\gamma')$ and thus indeed be
  identical.

  Since this is true for any illegal turn, all legal paths exiting
  from the same gate must have $g$-images with coinciding initial
  edge. This proves that the gate structure $\mathbf G$ is finer than
  or equal to $\mathbf G(g)$. The converse is true for any self-map
  $g$ that is a train track morphism with respect to a given gate
  structure $\mathbf G$, see Remark~\ref{classic-tts} (3).

\smallskip
\noindent
(b) Since $f$ is a train track morphism and hence has no contracted
edges, the above proved fact, that $Dg$ maps all edges in any given
gate to a single edge, is inherited by $D (f\circ g)$. Hence the
arguments from the previous paragraph are also true for $f\circ g$
instead of $g$, so that we obtain $\mathbf G(f \circ g) = \mathbf G$.
\end{proof}

\begin{rem}
\label{delicacy1}
(1) Note that the proof of the last lemma stays valid if the
hypothesis ``$g$ gate-stable'' is replaced by the weaker assumption
''$g$ gate structure morphism'' (see Definition~\ref{gate-morphism}).
In particular, by Lemma~\ref{periodic-gates} it suffices to assume
that all vertices of $\Gamma$ are periodic under the map $g$.

\smallskip
\noindent
(2) Note also that there is a delicacy in Statement~(b) of the
last lemma: The analogous statement for the composition $g \circ f$ is
in general wrong, unless one assumes that $f$ is a gate structure
morphism.
\end{rem}

\section{Legalizing maps for iwip automorphisms}
\label{legalizing-maps}

We will now concentrate on the situation of a classical train track
map $f: \Gamma \to \Gamma$, which is a train track morphism with
respect to the intrinsic gate structure $\mathbf G(f)$ on $\Gamma$,
see Remark~\ref{classic-tts}.  We assume furthermore that $f$ is a
homotopy equivalence so that it possesses a cancellation bound $C(f)$,
and that it satisfies the following expansion property:

\begin{defn}
\label{K-expanding}
For any constant $K \geq 1$ a train track morphism
$f: \Gamma \to \Gamma'$ is called {\em strongly $K$-expanding} if
every legal edge path $\gamma$ in $\Gamma$ of length $|\gamma|\geq K$
has $f$-image which is strictly longer:
\[
|f(\gamma)| \geq |\gamma| +1.
\]
\end{defn}

\begin{rem}
\label{discuss-expanding}
Expanding train track morphisms (see \S\ref{preliminaries}) are not
necessarily strongly $K$-expanding for some $K \geq 1$ (and
conversely), but it follows directly from the definitions that every
expanding train track morphism has a positive power which is strongly
$1$-expanding.
\end{rem}

We define the {\em minimal stretching factor} $\lambda^K_{min}(f)$ of
$f$ for legal paths of length $\geq K$ by:
\[
\lambda^K_{min}(f) := \min\{\frac{|f(\gamma)|}{|\gamma|} \mid
\gamma\text{ legal of length }|\gamma|\geq K\}
\]

We will now derive from any cancellation bound $C(f) \geq 0$ of a
strongly $K$-expanding train track morphism $f: \Gamma \to \Gamma$ an 
{\em expansion bound} $C(f)^+ \geq 0$:

\begin{lem}
\label{map-on-long-turns}
Let $f: \Gamma \to \Gamma$ be a train track map which possesses a
cancellation bound $C(f)$ and is strongly $K$-expanding for some $K
\geq 1$.  We define:
\[
C(f)^+ := \max(K,\frac{C(f)}{\lambda^K_{min}(f) - 1}).
\]
Let $C \geq C(f)^+$ be an integer.  Then the map $f$ induces a map
\[
f^{LT_C}: LT_C(\Gamma) \to LT_C(\Gamma), \,\, (\gamma,\gamma') \mapsto
f^{LT}(\gamma,\gamma')\chop^C\, .
\]
\end{lem}
\begin{proof}
  From the definition of the minimal stretching factor
  $\lambda^K_{min}(f)$ it follows that every legal path $\gamma$ of
  length $|\gamma| \geq C(f)^+\geq K$ is mapped by $f$ to a legal path
  of length $|f(\gamma)| \geq |\gamma| + C(f)$. Hence it follows from
  the definition of a cancellation bound $C(f)$ that any long turn
  $(\gamma,\gamma') \in LT_C(\Gamma)$ 
 is 
 $f$-long, and that its image
  long turn $f^{LT}(\gamma,\gamma')$ has length $\geq C$. Thus setting
  $(\gamma,\gamma') \mapsto f^{LT_C}(\gamma,\gamma')\chop^C$ defines
  indeed a well defined map $f^{LT_C}: LT_C(\Gamma) \to LT_C(\Gamma)$.
\end{proof}

\begin{prop}
\label{Nielsenpaths}
Let $f: \Gamma \to \Gamma$ be a train track map which possesses a
cancellation bound $C(f)$ and is strongly $K$-expanding for some $K
\geq 1$. Let $C \geq C(f)^+$.

Then for any $f^{LT_C}$-periodic illegal long turn $(\gamma, \gamma') \in
LT_C(\Gamma)$ the concatenation $\bar\gamma \circ \gamma'$ contains a
periodic INP as subpath.  Conversely, every periodic INP in $\Gamma$
can be prolonged on both sides so that the two legal branches give an
$f^{LT_C}$-periodic illegal long turn in $LT_C(\Gamma)$.
\end{prop}

\begin{proof}
  We know from Lemma~\ref{map-on-long-turns} that $f$ induces a well
  defined map $f^{LT_C}$ on the long turns in $LT_C(\Gamma)$. Assume
  now that for some integer $t \geq 1$ the long turn $(\gamma,
  \gamma') \in LT_C(\Gamma)$ is illegal and fixed by $(f^{LT_C})^t$.
  Then $\gamma$ is a subpath of $f^t(\gamma)$, and $\gamma'$ is a
  subpath of $f^t(\gamma')$.  Thus on both legal paths $\gamma$ and
  $\gamma'$ there must be a fixed point, which by the illegality of
  the turn must be different from the initial vertex of both, $\gamma$
  and $\gamma'$. We can define $\eta$ to be the path crossing over the
  illegal turn and connecting those two fixed points. Then
  $[f^t(\eta)] = \eta$, and since $\eta$ crosses over precisely one
  illegal turn, it follows that it is a periodic INP.

  Conversely, it follows from a standard calculation that the legal
  branches of any periodic INP $\eta$ can not be longer than
  $C(f)^+$. Thus they can be prolonged by legal paths so that this
  prolongation gives a long turn $(\gamma, \gamma') \in
  LT_C(\Gamma)$.
  Let now the integer $t \geq 1$ be such that $[f^t(\eta)] = \eta$. We
  consider the iterates of $(\gamma, \gamma')$ under $f^t$ and note
  that they all contain $\eta$ as subpath, in such a way that the
  illegal turn on $\eta$ coincides with the illegal turn formed by
  $(\gamma, \gamma')$ (and thus also by all of its
  $f^t$-iterates). Since $LT_C(\Gamma)$ is finite, eventually some
  such iterate $f^{kt}(\gamma, \gamma')$ will be
  $(f^{LT_C})^t$-periodic.  This shows the ``converse'' direction of
  the claim.
\end{proof}

\begin{cor}
\label{INP-alternative}
Let $f:\Gamma \to \Gamma$ be an expanding train track map that
represents an automorphism of $\FN$. Then precisely one of the
following is true:
\begin{enumerate}
\item[(a)] The map $f$ possesses a periodic INP, or
\item[(b)] every sufficiently high power of $f$ is legalizing for $\mathbf G(f)$.
\end{enumerate}
\end{cor}

\begin{proof}
  Since $f$ represents an automorphism, it possesses a cancellation
  bound.  Since $f$ is expanding, any sufficiently large power of $f$
  will be strongly $K$-expanding for $K = 1$ (see
  Remark~\ref{discuss-expanding}). Thus Proposition~\ref{Nielsenpaths}
  applies, so that, in case that $f$ does not possess a periodic INP,
  we can deduce that there is no illegal $f$-periodic long turn
  $(\gamma, \gamma') \in LT_C(\Gamma)$, for $C$ as in
  Proposition~\ref{Nielsenpaths}.

  It follows that after applying $f$ iteratively at least $k_0 :={\rm
    card}\, LT_C(\Gamma)$ times, any long turn $(\gamma, \gamma') \in
  LT_C(\Gamma)$ must have become legal. Since it stays legal under
  further iteration of $f$ (see Remark~\ref{legal-long}), every $f^k$ with $k \geq k_0$ must be
  legalizing.

  Clearly both (a) and (b) can not hold simultaneously.  Thus we have
  proved the desired dichotomy.
\end{proof}

One can derive from the last proof that the lower bound for the
exponent of $f$ needed in statement (b) of
Corollary~\ref{INP-alternative} can be efficiently calculated from the
train track map $f$ with not much effort. It turns out that it only
depends on the cancellation bound $C(f)$ and 
on 
the rank $N$ of the
free group $\FN$.

\begin{proof} [Proof of Theorem~\ref{thm:CL-main}]
  Since $g$ is assumed to be a train track morphism with respect to
  the gate structure $\mathbf G(f)$, then so must be $f\circ g$. Since
  $g$ induces an automorphism on $\pi_1\Gamma$, so does $f\circ g$, so
  that $f\circ g$ is a train track morphism which represents an
  automorphism of $\FN$.

  Recall from Section~\ref{preliminaries} that the transition matrix
  of $f\circ g$ is obtained as product $M(f\circ g) = M(f) \cdot
  M(g)$. Hence $M(f\circ g)$ inherits positivity from the assumed
  positivity of $M(f)$. Thus in particular $M(f\circ g)$ is primitive.

  As $g$ is gate-stable and legalizing for $\mathbf G(f)$, by
  Lemma~\ref{lem:legalizing-implies-intrinsec} the intrinsic gate
  structure $\mathbf G(f\circ g)$ is equal to $\mathbf G(f)$.

  By hypothesis, $f$ represents an iwip automorphism. From
  Proposition~\ref{prop:no-iwp-implies-irred}, we know that for any
  vertex $v$ the graph $Wh_{\mathbf G}^v(f)$ is connected. As $g$ is
  gate-stable, by Proposition~\ref{Wh-composition} the
  gate-Whitehead-graph $Wh^v_{\mathbf G}(f\circ g)$ must also be
  connected.

  We now observe from Proposition~\ref{composition-legalizing} that
  the composed map $f\circ g$ must be legalizing, which implies by
  Corollary~\ref{INP-alternative} that there are no periodic INPs for
  $f\circ g$.

  Thus the conditions (\ref{cond:primitive}),
  (\ref{cond:gwg-connected}) and (\ref{cond:no-inp}) of
  Proposition~\ref{JL-quote} are all satisfied for the map $f\circ g$,
  which hence must induce an iwip automorphism.
Corollary~\ref{index-lists} 
concludes the proof.
\end{proof}

\begin{rem}
\label{irred-up-to-iwip}
(1) We see from the above proof that the hypotheses in 
Theorem~\ref{thm:CL-main} can be weakened 
somewhat:
In the proof it is
never used that the automorphism represented by $f$ is iwip. It
suffices to assume that $M(f)$ is positive, and that the
gate-Whitehead-graph at every periodic vertex is connected.

\smallskip
\noindent
(2) We also don't use the fact that the gate structure on $\Gamma$ is
equal to $\mathbf G(f)$. It suffices to assume that $f$ and $g$ are
train track 
morphisms 
with respect to some fixed gate structure
$\mathbf G$, if in statement (\ref{concl:main-index}) of
Theorem~\ref{thm:CL-main} the list of gate indices at the $f$-periodic
vertices of $\Gamma$ is computed with respect to the gate structure
$\mathbf G$ (see Definition~\ref{gate-index}).  This is a consequence
of Lemma~\ref{lem:legalizing-implies-intrinsec}.

\smallskip
\noindent
(3) A careful analysis of the above proof and its various ingredients,
shows that the statement of Theorem~\ref{thm:CL-main} is valid as well
for the map $g\circ f$ in place of $f\circ g$, if one assumes in
addition that $f$ is a gate structure morphism (see
Remark~\ref{delicacy1} (2)).
\end{rem}

We'd like to remark here that part (2) of the previous remark gives
the possibility to produce, from a given train track map $f$ with a
fine gate structure, through properly choosing the legalizing
``perturbation map'' $g$, a variety of train track maps $f\circ g$
with coarser gate structures and thus, via
part~(\ref{concl:main-index}) of Theorem~\ref{thm:CL-main}, with
smaller index lists than $f$. A useful technology for the deliberate
production of such perturbation maps is described in the next section.

We conclude this section by passing to a larger set of product maps:

\begin{cor}
\label{monoid}
Let $\Gamma$ be a graph equipped with a gate structure $\mathbf G$,
and for any index $i$ of some index set $I$ let
$f_i:\Gamma \to \Gamma$ be a train track morphism with respect to
$\mathbf G$. Assume that each $f_i$ satisfies the following
properties:
\begin{enumerate}
\item\label{cond:monoid-intrinsic} The intrinsic gate structure satisfies
  $\mathbf G(f_i) = \mathbf G$.
\item\label{cond:monoid-matrix} The transition matrix $M(f_i)$ is positive.
\item\label{cond:monoid-connected} For any vertex $v$ of $\Gamma$ the graph $Wh_{\mathbf G}^v(f_i)$
  is connected.
\item\label{cond:monoid-no-inp} There is no periodic INP for $f_i$ in $\Gamma$.
\item\label{cond:monoid-gate-stable} The map $f_i$ is gate stable.
\end{enumerate}
Then there exist exponents $m_i \geq 1$ such that the
properties~(\ref{cond:monoid-matrix})~-~(\ref{cond:monoid-gate-stable})
hold for every element in the monoid generated by the $f_i^{m_i}$,
i.e. for any product
\[
f = f_{i_1}^{m_{i_1}} f_{i_2}^{m_{i_2}} \cdots f_{i_s}^{m_{i_s}}
\]
of the $f_i^{m_{i}}$ (but not their inverses!).  Furthermore, any such
map $f$ represents an iwip automorphism, the map $f$ is legalizing,
and the index list of $f$ is equal to the list of gate indices of
$\mathbf G$ at the vertices of $\Gamma$ with 3 or more gates.
\end{cor}

\begin{proof}
  By Corollary~\ref{INP-alternative} there exist exponents
  $m_i \geq 1$ such that each of the maps $f_i^{m_i}$ is
  legalizing. Furthermore, conditions (\ref{cond:monoid-matrix}),
  (\ref{cond:monoid-connected}) and (\ref{cond:monoid-no-inp}) ensure
  via Proposition~\ref{JL-quote} that $f_i$ represents an iwip
  automorphism of $\FN$.  Conditions~(\ref{cond:monoid-matrix}),
  (\ref{cond:monoid-connected}) and (\ref{cond:monoid-gate-stable})
  are inherited by products, if they are satisfied by every
  factor. The same is true for the property ``legalizing'', which
  implies condition~(\ref{cond:monoid-no-inp}).  By
  Lemma~\ref{lem:legalizing-implies-intrinsec}
  condition~(\ref{cond:monoid-intrinsic}) is a consequence of
  condition~(\ref{cond:monoid-gate-stable}) together with the property
  ``legalizing''.

  Hence all conditions for the factors $f$ and $g$ in
  Theorem~\ref{thm:CL-main} are satisfied for any of the maps
  $f_i^{m_i}$ as well as for any product $f$ as above. Thus the
  conclusion~(\ref{concl:main-index}) of Theorem~\ref{thm:CL-main}
  hold as well for $f$, which proves the last assertion in the
  statement of Corollary~\ref{monoid}.
\end{proof}

The above corollary admits a natural extension to a more involved
situation, where one considers simultaneously several graphs
$\Gamma_k$ with gate structures $\mathbf G_k$, as well as maps
$f_i: \Gamma_k \to \Gamma_{k'}$ which induce bijections on the
vertices with 3 or more gates, as well as bijections on the set of
their adjacent gates. This leads one directly to consider ``strata''
in Outer space, in analogy to strata in Teichm\"uller space as defined
by fixing the indices of the singularities of quadratic differentials,
see \cite{MS}.

\section{Legalizing Factory}\label{sec:legalizing-factory}

In this section we reduce the construction of a legalizing train track
morphism to the construction of a family of ``elementary'' train track
morphisms that each legalizes only a single illegal turn.

\begin{prop}\label{prop:legalizing-construction}
  Let $\Gamma$ be a graph equipped with a gate structure $\mathbf
  G$. Assume that there exists an integer $L\geq 1$ which satisfies:
\begin{enumerate}
\item\label{hyp:prop-legalizing-illeg-turn} For each illegal long turn
  $t = (\gamma, \gamma')$ of branch length $L$ there exists a train
  track morphism $g_t:\Gamma\to\Gamma$ such that $t$ is $g_t$-long and
  mapped by $g_t^{LT}$ to a legal long turn.
\item\label{hyp:prop-legalizing-expanding} There exists a train track
  morphism $h:\Gamma\to\Gamma$ which is strongly $K$-expanding for
  some $K\geq 1$.
\end{enumerate}
We assume furthermore that 
each 
of the above maps $g_t$ and $h$ has a
cancellation bound $C(g_t)$ or $C(h)$ respectively (which is true if
they induce automorphisms of $\pi_1(\Gamma)$).  Then there exists a
legalizing train track morphism $g: \Gamma \to \Gamma$ which is
obtained as a composition of the $g_t$ and $h$.
\end{prop}

\begin{proof}
  For each of the illegal long turns $t$ of $\mathbf G$ with branch
  length $L$ set $g'_t := h\circ g_t$, and observe that $g'_t$ is
  strongly $K$-expanding and inherits a cancellation bound $C(g'_t)$
  from $h$ and $g_t$.  Moreover, as $h$ is a train track morphism, it
  maps legal turns to legal turns, so that by
  hypothesis~(\ref{hyp:prop-legalizing-illeg-turn}) the long turn $t$
  is $g'_t$-long and mapped by $g'_t$ to a legal long turn.

  Let $C$ be the maximum of $L$ and of all the constants $C(g'_t)^+$,
  as defined in Lemma~\ref{map-on-long-turns} for any of the maps
  $g'_t$ via the cancellation bounds $C(g'_t)$ and the above constant
  $K$.  Then each $g'_t$ induces a well defined map ${g'_t}^{LT_C}$ on
  the set of long turns $LT_C(\Gamma)$.

  We can now build iteratively the legalizing train track morphisms we
  are looking for: Let $g_0$ be the identity map and
  $LT_C^{ill}(g_0) \subset LT_C(\Gamma)$ be the finite set of illegal
  long turns in $\Gamma$ of branch length $C$. We define iteratively
  graph maps $g_k: \Gamma \to \Gamma$ and nested subsets
\[
  LT_C^{ill}(g_k) \varsubsetneq LT_C^{ill}(g_{k-1}) \varsubsetneq
  \ldots \varsubsetneq LT_C^{ill}(g_0)
\]
by considering any turn $t^*$ in $LT_C^{ill}(g_k)$.  From the
iterative definition of $LT_C^{ill}(g_k)$ it follows that $t^*$ is
mapped by $g^{LT_C}_{k}$ to long a turn (of branch length $C$) which
is illegal. Let $t$ be the subturn of $g_k^{LT_C}(t^*)$ of branch
length $L \leq C$, which is of course also illegal.  We set
$g_{k+1} := g'_t\circ g_k$, and define $LT_C^{ill}(g_{k+1})$ to be the
set of illegal long turns in $\Gamma$ of branch length $C$ that are
mapped by $g_{k+1}^{LT_C}$ to an illegal long turn. Note that $g'_t$
was defined so that $t^*$ is mapped by $g_{k+1}^{LT}$ to a long turn
that contains the legal long turn $g'^{LT}_t(t)$ as subturn and is
therefore legal. Recall that if a long turn is mapped by $g_k^{LT_C}$
to a legal long turn, then, as $g'_t$ is a train track morphism, it is
also mapped by $g_{k+1}^{LT_C}$ to a legal long turn. In other words
$LT_C^{ill}(g_{k+1})\varsubsetneq LT_C^{ill}(g_k)$.

  From the finiteness of $LT_C^{ill}(g_{0})$ we deduce that after
  finitely many steps one gets $g_n$ with $LT_C^{ill}(g_n)=\emptyset$,
  which is equivalent to stating that $g=g_n$ is legalizing.
\end{proof}

\section{Stable indices, branching indices and gate indices}
\label{sec:index-branch-gate}

The content of this section is well known to the experts, or in close
proximity of well known facts; we assemble them here for the
convenience of the reader.  We will use some standard tools from
$\R$-trees and Outer space technology.  For background and terminology
the reader may consult \cite{Lu_conj-iwips}; further detail can be
found in \cite{CV} or \cite{Vog}.  We follow here mostly the original
source \cite{GJLL}.

\smallskip

For every expanding train track map $f: \Gamma \to \Gamma$ there
exists a non-negative real eigenvector $\vec v$ of the transition
matrix $M(f)$ which has real eigenvalue $\lambda > 1$, and any such
$\vec v$ determines an $\R$-tree $T = T(\vec v)$ (which in some cases
is called the {\em forward limit tree} and can be considered as
boundary point of Outer space $\CVN$).

The tree $T = T(\vec v)$ is obtained by choosing an arbitrary lift
$\tilde f: \tilde \Gamma \to \tilde \Gamma$ of the train track map $f$
to the universal covering $\tilde \Gamma$, and by defining $T$ to be
the metric space associated to the pseudo-metric $d_\infty$ on
$\tilde \Gamma$ which is the limit for $t \to \infty$ of the
decreasing sequence of pseudo-metrics
$d_t(x,y) := \frac{1}{\lambda^t}d_{\vec v}(\tilde f^t(x), \tilde
f^t(y))$.
Here the pseudo-metric $d_{\vec v}$ on $\tilde \Gamma$ is defined
through lifting the {\em $\vec v$-edge-lengths} of $\Gamma$ that are
explicitly given by the coefficients of the eigenvector $\vec v$.

As a consequence one obtains a canonical $\FN$-equivariant map
$i: \tilde \Gamma \to T$ which is {\em edge-isometric} with respect to
the pseudo-metric $d_{\vec v}$, i.e. every edge $e$ of $\tilde \Gamma$
is mapped by $i$ isometrically to its image $i(e) \subset T$.  The map
$\tilde f$ also induces directly a homothety $H: T \to T$ with
stretching factor $\lambda$, and one obtains the following
``commutative diagram'':
\begin{equation}
\label{comm-diag}
H \circ i = i \circ \tilde f 
\end{equation}
The map $i$ maps legal paths in $\tilde \Gamma$ isometrically to
segments in $T$. On the other hand, any path $\tilde \eta$ in
$\tilde \Gamma$ which is the lift of a periodic INP in $\Gamma$ is
folded by $i$ completely to a single segment, which is the isometric
image of any of the two legal branches of $\tilde \eta$.

It follows from standard train track arguments (see for example
Section~3 of \cite{KL}) that for any path $\gamma$ in $\Gamma$ a
sufficiently high $f$-iterate $f^t(\gamma)$ is homotopic
rel. endpoints to a {\em pseudo-legal} path, i.e. a legal
concatenation of legal paths and periodic INPs.  It follows that for
any two points $x, y \in \tilde \Gamma$ one has $i(x) = i(y)$ if and
only if after iterating $\tilde f$ sufficiently many times the
geodesic path $\tilde \gamma$ in the tree $\tilde \Gamma$ which joins
$\tilde f^t(x)$ to $\tilde f^t(y)$ is a legal concatenation of legal
subpaths and lifts of periodic INPs, where the legal subpaths only
use edges with $d_{\vec v}$-length 0. In particular, we see that the
absence of INPs for $f$ implies directly that the $\FN$-action on $T$
is free, if all the exponents of the the eigenvector $\vec v$ are
positive.  The latter is known if the transition matrix $M(f)$ is
primitive, and hence always true if the expanding train track map $f$
represents an iwip automorphism.  Furthermore, the North-South result
of the $\phi$-action on the closure of $\CVN$ proved in \cite{LL4}
yields:

\begin{prop}
\label{forward-limit-tree}
For any iwip automorphism $\phi$ the forward limit tree
$T = T(\vec v)$ is well defined up to uniform rescaling of the metric,
and in particular does not depend on the expanding train track
representative $f: \Gamma \to \Gamma$ and its primitive transition
matrix $M(f)$ with Perron-Frobenius eigenvector $\vec v$.

If $\Gamma$ doesn't contain any 
non-trivial loop which is a legal concatenation of periodic INPs,
then the $\FN$-action on
$T$ by isometries is free.  
This conclusion is in particular true if there is no periodic INP in $\Gamma$.
\qed
\end{prop}

As a direct consequence of the above described construction of $T$
from $\tilde \Gamma$ by means of the eigenvector $\vec v$ one has the
following fact, which is well known to the experts (see
\cite{GJLL}, \cite{HM}, \cite{Lu_conj-iwips}, or, for much detail, \S~7
of \cite{KL-7steps-v3}). Recall that a {\em direction} at a point
$P \in T$ is a connected component of $T \smallsetminus \{P\}$.

\begin{prop}
\label{gates-branch-points}
Let $f: \Gamma \to \Gamma$ be an expanding train track map, and let
$\vec v$ be an eigenvector of $M(f)$ with eigenvalue $\lambda > 1$.
Assume that $\vec v$ has only positive coefficients, and let
$T = T(\vec v)$ the corresponding forward limit tree.

\smallskip
\noindent
(1) If there is no periodic INP for $f$, then the map $i$ restricts to
an $\FN$-equivariant bijection $i_V$ between essential vertices $v_k$
of $\tilde \Gamma$ (i.e. lifts of 
$f$-periodic 
vertices of $\Gamma$ with 3 or more
gates) 
on one hand, and branch points $i(v_k)$
of $T$ on the other. This bijection extends to a canonical bijection
between the gates at any $v_k$ and the directions at $i(v_k)$ (where a
gate $\mathfrak g_j$ is mapped to the direction that contains the open
segments $i(\inter e_i)$ for any edge $e_i$ in $\mathfrak g_j$).

\smallskip
\noindent
(2) [Not used in the sequel.]  If $f$ possesses periodic INPs, then
the endpoints (assumed to have been made into vertices) of any 
such 
periodic INP 
$\eta$ 
have to be considered as equivalent, and
for these endpoints one has to identify those 
two 
gates which contain the
two branches of 
$\eta$. 
Then we get the precisely analogous statement
as in the ``no INP'' case (1), except that the preimage 
of a
branch point in $T$ will now be 
the lift of an 
$f$-periodic equivalence 
class 
of
vertices in 
$\Gamma$ 
with (after the above identification) 3 or
more gates.  
\qed
\end{prop}

If $X$ is a topological space, provided with a left action of a group
$G$ by homeomorphisms, we say that a map $F: X \to X$ {\em represents}
an automorphism $\Phi \in \Aut(G)$ if for all $x \in X$ and all
$g \in G$ one has:
\[
\Phi(g) \cdot F(x) = F(g \cdot x)
\]
This applies in particular to the special case where $X$ is the
universal covering of the quotient space $X/G$ with $\pi_1(X) = G$,
and $F$ induces a homeomorphism $f: X/G \to X/G$.  In this case, if
$F$ represents $\Phi$, then $f$ induces on $\pi_1(X) = G$ the outer
automorphism $\phi$ defined by $\Phi$.

\smallskip

It follows from the equality (\ref{comm-diag}) and the
$\FN$-equivariance of the map $i$ that any lift $\tilde f$ of the
train track map $f$ represents the same automorphism
$\Phi \in \Aut(\FN)$ as the associated homothety $H: T \to T$, where
$\Phi$ induces (by a the previous paragraph) the outer automorphims
$\phi$ that is represented by the train track map $f$. Since the
stretching factor of $H$ satisfies $\lambda > 1$, it follows that $H$
has precisely one fixed point $Q$ which is either contained in $T$, or
it lies in the metric completion $\bar T$ of $T$ (where we use the
canonical extension of $H$ to $\bar T$).

\smallskip

Let now $\Phi' \in \Aut(\FN)$ be a second lift of $\phi$, and assume
that $\Phi'$ is isogredient to $\Phi$, i.e.
$\Phi' = \iota_w \circ \Phi \circ \iota_{w^{-1}} = \iota_w \circ
\iota_{\Phi(w)^{-1}} \circ \Phi$
for some $w \in \FN$, where $\iota_v: \FN \to \FN$ denotes the
conjugation $u \mapsto v u v^{-1}$.  Assume furthermore that the lift
$\tilde f': \tilde \Gamma \to\tilde \Gamma$ of $f$ and the homothety
$H': T\to T$ both represent the automorphism $\Phi'$.  Then we obtain
$\tilde f' = w \tilde f w^{-1}$ and $H' = w H w^{-1}$, and thus deduce
for the fixed point $Q'$ of $H'$ the equality $Q' = w Q$.

Conversely, if $H' = u H$ is the homothety of $T$ which represent some
lift $\Phi' = i_u \Phi$ of $\phi$, then, if $H'(Q) = Q$, the action of
$u$ on $T$ must fix the point $Q$. Thus, if the $\FN$-action on $T$ is
free, then one deduces $u = 1$.

As a consequence one gets a natural injective map from the
isogredience classes of lifts $\Phi_k$ of $\phi$ into the set of
$\FN$-orbits of points in $\bar T$, given by the fixed point $Q_k$ of
the associated homothety $H_k: T \to T$ that represents $\Phi_k$.

\smallskip

In \cite[Theorem~2.1~(3) and Proposition~4.4]{GJLL}, the following has
been proved (we only cite the easy case where $\Stab (Q)$ is trivial):

\begin{prop}
\label{prop-4-4-GJLL}
Let $\phi \in \Out(\FN)$ be an iwip automorphism, and let
$T = T(\vec v)$ be the forward limit tree of $\phi$, given as above by
some eigenvector $\vec v$ with eigenvalue $\lambda > 1$ of the
transition matrix $M(f)$ of a train track representative
$f: \Gamma \to \Gamma$ of $\phi$. We assume that there is no periodic
INP in $\Gamma$, so that the $\FN$-action on $T$ is free.

Let $H: T \to T$ be the homothety (with stretching factor $\lambda$)
that represents some lift $\Phi \in \Aut(\FN)$ of $\phi$, and assume
$\Phi$ has index ${\rm Ind}(\Phi) > 0$. Then the fixed point $Q$ of
$H$ is contained in $T$, and $Q$ is a branch point of $T$.  There is a
natural injection $i_Q$ from the set of attracting fixed points of
$\Phi$ on $\partial \FN$ to the set of directions at $Q \in T$. The
image of $i_Q$ is precisely the set of those directions that are fixed
by $H$.  \qed
\end{prop}

From the last sentence of this proposition we see that replacing
$\phi$ (and hence also $\Phi$ and $H$) by a positive power will
increase the image set of the map $i_Q$.  From Gaboriau and
Levitt~\cite{GL} finiteness result, one knows that for any $T$ with
free $\FN$-action every branch point has only finitely many
directions, so that a suitable positive power of $H$ will indeed fix
every direction of $T$ at $Q$.

The same finiteness result \cite{GL} also implies that there are only
finitely many $\FN$-orbits of branch points in $T$.  Thus through
possibly replacing $\phi$ by a further positive power we can assume
that the associated homothety $H$ of $T$ fixes every $\FN$-orbit of
branch points of $T$.  Thus for any branch point $Q'$ of $T$ there is
a suitable element $u \in \FN$ such that the homothety $H' = u H$ has
$Q'$ as fixed point (pick $u$ such that $H(Q') = u^{-1}Q'$). Now we
``perturb''
$\Phi$ and $\tilde f$ correspondingly to obtain
$\Phi' = \iota_u \circ \Phi$ and $\tilde f' = u \tilde f$, and, if
need be, we pass to another common positive power, so that the
homothety $H'$ fixes every direction of $T$ at $Q'$. As a consequence,
the map $i_{Q'}$ from Proposition~\ref{prop-4-4-GJLL} becomes a
bijection between the attracting fixed points of $\Phi'$ on
$\partial \FN$ and the directions of $T$ at $Q'$.

\smallskip

Thus we obtain:

\begin{prop}
\label{branch-points-index}
Let $\phi \in \Out(\FN)$ be an iwip automorphism that has an expanding
train track representative $f: \Gamma \to \Gamma$ without periodic
INPs, and let $T$ be its forward limit tree.

Then, for 
some integer $t \geq 1$, 
there exists a natural
bijection between on one hand the isogredience classes of
representatives $\Phi_k$ of $\phi^t$ which satisfy $\ind \Phi_k \geq
\frac{1}{2}$, and on the other hand the $\FN$-orbits $\FN \cdot Q_k$
of branch points $Q_k$ of $T$.

This correspondence extends further to a bijection between the set of
attractors for the induced $\Phi_k$-action on $\partial \FN$, and the
set of directions of $T$ at $Q_k$.  

The assertion remains valid if the integer $t$ is replaced by any positive integer $t' = k t \in \N$.
\qed
\end{prop}

Recall from the Introduction that the stable index list of an
automorphism $\phi \in \Out(\FN)$ without non-trivial periodic
conjugacy classes is given by the maximal (decreasing) list of indices
$\ind(\Phi_k) = \frac{a(\Phi_k)}{2} -1$ for representatives
$\Phi_k \in \Aut(\FN)$ of a suitable positive power of $\phi$ that are
pairwise non-isogredient, where $a(\Phi_k)$ denotes the number of
attractors of $\Phi_k$ on $\partial \FN$.

Thus Proposition~\ref{branch-points-index} shows that for an iwip
automorphism $\phi$, assumed to have an expanding train track
representative $f: \Gamma \to \Gamma$ without periodic INPs, the
stable index list of $\phi$ agrees with the {\em branching index list}
of $T$, i.e. the maximal (decreasing) sequence of values
$\ind(Q_k) := \frac{b(Q_k)}{2}-1$ for branch points $Q_k$ in distinct
$\FN$-orbits, where $b(Q_k)$ denotes the number of directions of $T$
at $Q_k$. On the other hand, we obtain directly from part (1) of
Proposition~\ref{gates-branch-points} that the branching index list of
$T$ agrees with the gate index list for the map
$f: \Gamma \to \Gamma$, as defined in Definition~\ref{gate-index}.
Thus we obtain:

\begin{cor}
\label{index-lists}
Let $\phi \in \Out(\FN)$ be an iwip automorphism, and let
$f: \Gamma \to \Gamma$ be an expanding train track map which
represents $\phi$ and which doesn't have any periodic INP.

Then the stable index list of the automorphism $\phi$ agrees with the
gate index list of the map $f$.  \qed
\end{cor}

In fact, we see from the details of the above correspondences that,
after raising $\phi$ and $f$ to a 
suitable 
common positive power, there is a
natural bijection between the essential vertices of $f$ and the
isogredience classes of $\phi$, which for any essential vertex $v_k$
and any representative $\Phi_k$ of the corresponding isogredience
class extends further to a bijection between the attracting fixed
points of $\Phi_k$ on $\partial \FN$ and the gates at $v_k$. This last
bijection can be seen concretely by considering the unique eigenray
$\rho$ defined by any gate at $v_k$, its lift to an eigenray
$\tilde \rho$ of a (suitably chosen) lift $\tilde f_k$ of $f$ for the
universal covering $\tilde \Gamma$, and the image
$i(\tilde \rho) \in T$, which is an eigenray of the associated
homothety $H_k$ which represents $\Phi_k$. Under the canonical
$\FN$-equivariant identification
$\partial \FN = \partial \tilde \Gamma$ the ray $\tilde \rho$
represents an attractor of the $\Phi_k$-action on $\partial \FN$, and
conversely, every attractor for $\Phi_k$ comes from such an eigenray.

\bibliographystyle{alpha}

\end{document}